\documentclass[11pt]{article}

\newcommand{\ds}{\displaystyle}
\newcommand{\fracp}[2]{\frac{\partial #1}{\partial #2}}

\newcommand{\be}{\begin{equation}}
\newcommand{\ee}{\end{equation}}

\newcommand{\bi}{\begin{itemize}}
\newcommand{\ei}{\end{itemize}}

\newcommand{\bea}{\begin{eqnarray}}
\newcommand{\eea}{\end{eqnarray}}
\newcommand{\bean}{\begin{eqnarray*}}
\newcommand{\eean}{\end{eqnarray*}}

\newcommand{\ba}{\begin{array}}
\newcommand{\ea}{\end{array}}

\renewcommand{\v}[1]{\mbox{\boldmath $#1$}}

\newcommand{\nn}{\nonumber}
\newcommand{\ov}{\overline}

\newcommand{\rit}{\hbox{\it I\hskip -2pt R}}
\def\srit{\hbox{\scriptsize{\it I\hskip -1.5pt R}}}

\newcommand{\sit}{\hbox{\it I\hskip -5pt S}}

\newcommand{\eps}{\varepsilon}

\renewcommand{\div}{\nabla\cdot}

\newcommand{\intom}{\int_{\Omega}}

\newcommand{\evec}{\mathbf{E}}

\newcommand{\bvec}{\mathbf{B}}

\newcommand{\vvec}{\mathbf{v}}
\newcommand{\jvec}{\mathbf{J}}
\newcommand{\nvec}{\mathbf{n}}
\newcommand{\xvec}{\mathbf{x}}
\newcommand{\uvec}{\mathbf{u}}

\usepackage{theorem}

\theoremheaderfont{\scshape}
\newtheorem{thm}{Theorem}[section]
\newtheorem{lem}[thm]{Lemma}

\theorembodyfont{\normalfont}
\newtheorem{rem}{Remark}[section]

\newenvironment{proof}{\noindent{\it Proof. }}{\hfill\rule{2mm}{2mm}\vskip3mm \par}

\newcommand{\verbatimfile}[1]{\expandafter\beginverb\input{#1}}
\newcommand{\beginverb}{\begin{verbatim}}

\newcommand{\Section}{\setcounter{equation}{0} \section}

\def\mvec{{\cal M}}
\def\nvec{{\cal N}}
\def\RRR{{\cal R}}
\def\BB{{\cal B}}
\def\rvec{{\bf r}}
\def\ovr{{\bf \overline r}}
\def\calq{{\cal Q}}
\def\calo{{\cal O}}

\def\intrx{\int_{\srit_x^3}}
\def\intrv{\int_{\srit_v^3}}
\def\intrxrv{\int_{\srit_x^3\times\srit_v^3}}
\def\intq{\int_\calq}

\def\basu{{\bf e}_1}
\def\basd{{\bf e}_2}
\def\bast{{\bf e}_3}

\begin{document}
{\LARGE \bf
Homogenization of the Vlasov equation and of the Vlasov - Poisson
system with a strong external magnetic field}

\bigskip
\noindent Emmanuel Fr\'enod\\
LMAM, Universit\'e de Bretagne Sud \\
Centre Universitaire de Vannes\\
 1 rue de la loi, F-56000 Vannes\\
E-mail: Emmanuel.Frenod@univ-ubs.fr

\medskip
\noindent Eric Sonnendr\"ucker\\
CNRS, Institut Elie Cartan, Universit\'e Henri Poincar\'e - Nancy 1\\
B.P. 239\\
F-54506 Vand{\oe}uvre-l\`es-Nancy Cedex\\
E-mail: Eric.Sonnendrucker@iecn.u-nancy.fr

\bigskip 

{\bf Abstract:} Motivated by the difficulty arising in the numerical
  simulation of the movement of charged particles in presence of a
  large external magnetic field, which adds an additional time scale
  and thus imposes to use a much smaller time step, we perform in this
  paper a homogenization of the Vlasov equation and the Vlasov-Poisson
  system which yield approximate equations describing the mean
  behavior of the particles. The convergence proof is based on the two
  scale convergence tools introduced by N'Guetseng and Allaire. We
  also consider the case where, in addition to the magnetic field, a
  large external electric field orthogonal to the magnetic field and
  of the same magnitude is applied.


{\bf Abbreviated title:} Homogenization of the Vlasov equation

\newpage
\Section{Introduction}
In many kind of devices involving charged particles, like electron guns, 
diodes or tokamaks, a large external magnetic field needs to be applied in order
to keep the particles on the desired tracks. In Particle-In-Cell (PIC)
simulations of such devices, this large external magnetic field obviously needs
to be taken into account when pushing the particles. However, due to the
magnitude of the concerned field this often adds a new time scale to the
simulation and thus a stringent restriction on the time step. In order to get rid
of this additional time scale, we would like to find approximate equations, where
only the gross behavior implied by the external field would be retained and
which could be used in a numerical simulation.
 
The trajectory of a particle in  a constant magnetic field $\bvec$ is a
helicoid along the magnetic field lines with a radius proportional to the
inverse of the magnitude of $\bvec$. Hence, when this field becomes very large
the particle gets trapped along the magnetic field lines. However due to the
fast oscillations around the apparent trajectory, its apparent velocity is
smaller than the actual one. This result has been known for some time 
as the "guiding center" approximation, and the
link between the real and the apparent velocity is well known in terms of
$\bvec$. 
We refer to Lee \cite{lee:1983} and Dubin et {\it al} \cite{dubin/etal:1983}
for a complete physical viewpoint review on this subject.
In the case of a cloud of particles, the movement of which is
described by the Vlasov-Poisson equations, the situation is less clear as in
the one particle case because of the non linearity of the problem. In
particular, the question of knowing if the mutual influence of the particles
can be expressed in terms of their apparent motion or if the oscillation
generates additional effects is important.

In this paper, we deduce the "guiding center" approximation in
the framework of partial differential equations via an homogenization
process on the linear Vlasov equation.
Then, we show that in a cloud of particles the mutual influence of the 
particles can be expressed in term of their apparent motion without
any additionnal terms. This is provided applying an homogenization
process on the Vlasov-Poisson system, similar to the one used for the
linear Vlasov equation and using the regularity of the charge density.

Hence, we apply first a homogenization
process to the linear Vlasov equation with a strong and constant external
magnetic field. In other
words  for a constant vector $\mvec \in \sit^2$, 
we consider the following equation:
\be\label{vlasov}\left\{\ba{l}
\ds\fracp{f^{\eps}}{t}+\vvec\cdot\nabla_xf^{\eps}+(\evec^{\eps}+\vvec
\times(\bvec^{\eps}+\frac{\mvec}{\eps}))
\cdot\nabla_vf^{\eps} =0,\\
\ds f^\eps_{|t=0}=f_0.
\ea\right.\ee
In this equation $f^{\eps}\equiv f^{\eps}(t,\xvec,\vvec)$ 
with $t\in [0,T)$, for any $T\in \rit^+$,
$\xvec\in\rit^3_x$ and $\vvec\in\rit^3_v$. For convenience, we introduce the
notations $\Omega=\rit^3_x\times\rit^3_v,$ $\calo=[0,T)\times\rit_x^3$
and $\calq=[0,T)\times\Omega$.
The initial data satisfies
\be\label{ic1}
\ds f_0\geq 0,\quad 0<\intom f_0^2\,d\xvec\,d\vvec < \infty.
\ee
The electric field $\evec^{\eps}\equiv \evec^{\eps}(t,\xvec)$ and the magnetic
field $\bvec^{\eps}\equiv \bvec^{\eps}(t,\xvec)$ are defined on
$\calo,$ both bounded in $L^\infty(0,T;L^2(\rit^3_x))$ and satisfy
\be\label{cvge}
\evec^\eps\rightarrow \evec \mbox{ in } 
L^\infty(0,T;L^2_{loc}(\rit^3_x))\mbox{ strong,}
\ee 
and
\be\label{cvgb}
\bvec^\eps\rightarrow \bvec \mbox{ in } L^\infty(0,T; 
L^2_{loc}(\rit^3_x))\mbox{ strong,}
\ee
for any $T\in \rit^+.$
\newline
With those conditions, for every $\eps$ there exists a unique solution of the equation
(\ref{vlasov}) $f^\eps\in
L^\infty(0,T;$ $L^2(\Omega))$. 
We characterize here the
equation satisfied by the limit $f$ (in some weak topologies) of
$(f^\eps)_\eps$.
\newline The first main result of the paper is the following:
\begin{thm}\label{Th1.1}
Under assumptions (\ref{ic1})-(\ref{cvgb}), the sequence
$(f^\eps)_\eps$ of solutions of the
Vlasov equation (\ref{vlasov}) satisfies, for any $T\in \rit^+$,
$$f^\eps\rightharpoonup f \mbox{ in } L^\infty(0,T;L^2(\Omega)) \mbox{
weak}-*.$$
Moreover, denoting for any vector $\vvec,$ $\vvec_\parallel=$
$(\vvec\cdot\mvec)\mvec,$
$f$ is the unique solution of:
\be\label{vlasovlim}\left\{\ba{l}
\ds\fracp{f}{t}+\vvec_\parallel\cdot\nabla_xf+(\evec_\parallel+\vvec
\times\bvec_\parallel)\cdot\nabla_vf =0, \\
\ds f_{|t=0}=\frac{1}{2\pi} \int_0^{2\pi}f_0(\xvec,\uvec(\vvec,\tau))\,
d\tau,
\ea\right.\ee
where $\uvec(\vvec,\tau)$ is the rotation of angle $\tau$ around $\mvec$
applied to $\vvec$ (see (\ref{defu}) for more details).
\end{thm}
The deduction of this Theorem uses the notion of two scale convergence 
introduced by N'Guetseng \cite{nguetseng:1989} and Allaire 
\cite{allaire:1992}. This convergence result is the following:
\begin{thm}\label{N'GA}(N'Guetseng \cite{nguetseng:1989} and Allaire 
\cite{allaire:1992})
If a sequence $f^\eps$ is bounded in $L^\infty(0,T;L^2(\Omega)),$
then for every period $\theta$, there exists a
$\theta$-periodic profile $F_\theta(t,\tau,\xvec,\vvec)\in L^\infty(0,T;
L_\theta^\infty(\rit_\tau; L^2(\Omega)))$ such that for all
$\psi_\theta(t,\tau,\xvec,\vvec)$ regular, with compact support with respect
to $(t,\xvec,\vvec)$ and $\theta$-periodic with respect to $\tau$ we have,
up to a subsequence,
\be\label{0tsv1}
\intq f^\eps\psi_\theta^\eps\,dt\,d\xvec\,d\vvec\rightarrow \intq
\int_0^\theta F_\theta\psi_\theta\,d\tau\,dt\,d\xvec\,d\vvec.\ee
Above, $L_\theta^\infty(\rit_\tau)$ stands for
the space of functions being $L^\infty(\rit)$ and being $\theta$-periodic
and $\psi_\theta^\eps\equiv\psi_\theta(t,\frac{t}{\eps},\xvec,\vvec)$.

The profile $F_\theta$ is called the $\theta$-periodic two scale limit of
$f^\eps$ and 
the link between $F_\theta$ and the weak$-*$ limit $f$ is given by
\be\label{tsv2}
\int_0^\theta F_\theta(t,\tau,\xvec,\vvec)\,d\tau
= f(t,\xvec,\vvec).\ee 
\end{thm}
This result has been used with success in the context of homogenization of
transport equation with periodically oscillating coefficients by
E \cite{e:1992} and Alexandre and Hamdache \cite{alexandre/hamdache}
and in the context of kinetic equations with strong and periodically 
oscillating coefficients in Fr\'enod \cite{frenod:1994}
and Fr\'enod and Hamdache \cite{frenod/hamdache:1996}.
\newline Here, since the strong magnetic field induces periodic oscillations
in $f^\eps,$ the two scale limit describes well its asymptotic behavior.
Therefore, it is the right tool to tackle homogenization of equation 
(\ref{vlasov}), and Theorem \ref{Th1.1} is a consequence of the following
result concerning the two scale limit of $f^\eps.$
\begin{thm}\label{Th1.2}
Under assumptions (\ref{ic1})-(\ref{cvgb}),
the $2\pi$-periodic two scale limit $F \in L^\infty(0,T;$
$L_{2\pi}^\infty(\rit_\tau; L^2(\Omega)))$ of $f^\eps$ is 
the unique solution of
\be\label{vlastwoscale}\left\{\ba{l}
\ds \fracp{F}{\tau} + (\vvec\times\mvec) \cdot\nabla_v F=0,\label{contr}
\vspace{1mm}\\
\ds \fracp{F}{t}+\vvec_\parallel\cdot\nabla_x F+(\evec_\parallel+\vvec
\times\bvec_\parallel)
\cdot\nabla_v F =0, \\
\ds F_{|t=0}=\frac{1}{2\pi}f_0(\xvec,\uvec(\vvec,\tau)),
\ea\right.\ee
where $\uvec(\vvec,\tau)$ is the rotation of angle $\tau$ around $\mvec$
applied to $\vvec$ (see (\ref{defu}) for more details).
\end{thm}
\begin{rem}
This last Theorem can be generalised to the case of a non uniform
strong magnetic field, see Remark \ref{ZobiRem}.
\end{rem}
\begin{rem}
The method we develop enables also to deduce the homogenized
equations  when a strong external electric field $\frac\nvec\eps$
is added in equation (\ref{vlasov}). In this case we get, in addition
a $\nvec\times\mvec$ drift.This drift shares a lot with the one used
in the ``guiding center'' approximation.
See Theorems \ref{thth3.1} and \ref{thth3.2} for more details.  
\end{rem}
Secondly, as a relatively direct application of the former results, we may
characterize the asymptotic behavior of a sequence 
$(f^\eps, \evec^\eps)$ of solutions of the following Vlasov-Poisson
system

\be\label{vlaspoiss}\left\{\ba{l}
\ds\fracp{f^{\eps}}{t}+\vvec\cdot\nabla_xf^{\eps}+(\evec^{\eps}+\vvec
\times\frac{\mvec}{\eps})
\cdot\nabla_vf^{\eps} = 0, \\
\ds f^{\eps}_{|t=0}=f_0, \\
\evec^{\eps}=-\nabla u^{\eps}, \;\; -\Delta u^{\eps}= \rho^{\eps}, 
\ea\right.\ee
with
$$\rho^{\eps}=\intrv f^{\eps}\,d\vvec,$$
where $f_0$ is assumed to satisfy
\be\label{ic7}
\ds f_0\geq 0,\quad  f_0\in L^1\cap L^2(\Omega),
\mbox{ and }
0<\intom f_0(1+|v|^2)\,d\xvec\,d\vvec < \infty.
\ee
The second main result of the paper is the following:
\begin{thm}\label{Th1.3}
Under assumption (\ref{ic7}), for each $\eps,$ there exists
a solution $(f^\eps,\evec^\eps)$ of (\ref{vlaspoiss})
in $L^\infty(0,T;L^1\cap L^2(\Omega) ) \times 
L^\infty(0,T;W^{1,\frac75}(\rit^3_x)$
for any $T\in \rit^+.$
Moreover this solution is bounded in 
$L^\infty(0,T;L^1\cap L^2(\Omega) ) \times 
L^\infty(0,T;W^{1,\frac75}(\rit^3_x))$ independently on $\eps.$ 

\noindent If we consider a sequence $(f^\eps,\evec^\eps)$ of such solutions,
extracting a subsequence, we have
$$f^\eps\rightharpoonup f \mbox{ in } L^\infty(0,T;L^2(\Omega)) \mbox{
weak}-*,$$
$$\evec^\eps\rightarrow \evec \mbox{ in } 
L^\infty(0,T;L^2_{loc}(\rit^3_x))\mbox{ strong,} $$ 
and, denoting for any vector $\vvec,$ $\vvec_\parallel=$
$(\vvec\cdot\mvec)\mvec,$ the limit $(f,\evec)$ satisfies 
\be\label{vlaspoisslim}\left\{\ba{l}
\ds\fracp{f}{t}+\vvec_\parallel\cdot\nabla_xf+
\evec_\parallel\cdot\nabla_vf =0, \\
\ds f_{|t=0}=\frac{1}{2\pi} \int_0^{2\pi}f_0(\xvec,\uvec(\vvec,\tau))\,
d\tau,\\
\ds\evec=-\nabla u, \;\; -\Delta u=\rho, 
\ea\right.\ee
with
\be\label{pinoderch1}
\rho=\intrv f \,d\vvec,
\ee
where $\uvec(\vvec,\tau)$ is the rotation of angle $\tau$ around $\mvec$
applied to $\vvec$ (see (\ref{defu}) for more details).
\end{thm}
This Theorem is a consequence of the following result.
\begin{thm}\label{Th1.4}
Under assumption (\ref{ic7}),
extracting a subsequence, the $2\pi-$periodic two scale limit 
$F\in L^\infty(0,T;L^\infty_{2\pi}(\rit_\tau;L^1\cap L^2(\Omega))) $ of 
$f^\eps$ is solution of
\be\label{vlaspoisstwoscale}\left\{\ba{l}
\ds \fracp{F}{\tau} + (\vvec\times\mvec) \cdot\nabla_v F=0,\label{contr1}
\vspace{1mm}\\
\ds \fracp{F}{t}+\vvec_\parallel\cdot\nabla_x F+\evec_\parallel
\cdot\nabla_v F =0, \\
\ds F_{|t=0}=\frac1{2\pi}f_0(\xvec,\uvec(\vvec,\tau))\\
\ds\evec=-\nabla u, \;\; -\Delta u=\ov\rho, 
\ea\right.\ee
with $\ov\rho$ given by 
\be\label{defrhobar}
\ov\rho=\intrv F \,d\vvec.
\ee
Moreover, $\ov\rho$ does not depend on $\tau$ and
\be\label{rhorhobar}
\ov\rho=\rho.
\ee
\end{thm} 
This Theorem provides a rigorous justification of the procedures called
subcycling and orbit averaging that are often used in Particle-In-Cell
simulations of the Vlasov-Poisson equations in order to reduce the cost 
of the simulation. This procedure consists in advancing the particles, 
which provide an approximate solution of the Vlasov equation, 
on a smaller time step than the one
used to advance the solution of the Poisson equations,  see
\cite{cohen:1985} for an overview. 

\begin{rem}
We may add in
(\ref{vlaspoiss}) an external magnetic field $\bvec^\eps$ strongly 
converging to
$\bvec.$ The Theorems \ref{Th1.2} and \ref{Th1.1} will also apply in
the same way leading to equation (\ref{vlaspoisstwoscale})
with (\ref{vlaspoisstwoscale}.b) replaced by
$$
\fracp{F}{t}+\vvec_\parallel\cdot\nabla_x F+(\evec_\parallel+\vvec
\times\bvec_\parallel)
\cdot\nabla_v F =0,
$$
and to (\ref{vlaspoisslim}) with (\ref{vlaspoisslim}.a) replaced by
$$
\ds\fracp{f}{t}+\vvec_\parallel\cdot\nabla_xf+
(\evec_\parallel+\vvec \times\bvec_\parallel)\cdot\nabla_vf =0.
$$
\end{rem}
\begin{rem}
the results of this paper are connected to the ones of 
Grenier \cite{grenier:1997,grenier2:1997} and Schochet
\cite{schochet:1994} concerning the pertubation of 
hyperbolic systems by a $\frac1\eps-$depending linear
operator inducing fast oscillations in time.
In those work, the fast  oscillations are canceled applying
a ``reverse-oscillating'' operator to the familly of solutions
of the $\eps-$depending equation whose limit is seeked.
\newline Yet, in the present work, the oscillations are 
treated using an ad-hoc class of oscillating test functions.

\end{rem}

From now on, and with no loss of generality, we set
$\mvec=\basu,$
where $\basu$ is the first vector of the basis $(\basu,\basd,\bast)$ of
$\rit^3$. 

The paper is organized as follow: In section \ref{secHVE}, we provide the
homogenization of the Vlasov equation (\ref{vlasov}). For this purpose, we
first deduce the 
equation satisfied by the two scale limit $F$ of $f^\eps$ and then we obtain
the homogenized equation of Theorem \ref{Th1.1}. 
\newline In section \ref{secSME} we study the case when a strong 
electric field, orthogonal to the strong magnetic field, is added in
the Vlasov equation.
\newline The last section is devoted to the homogenization of the 
Vlasov-Poisson system (\ref{vlaspoiss}). This is an application of the 
result proved in section \ref{secHVP} once the regularity of $\rho^\eps$
implying the strong convergence of $E^\eps$ is exhibited.

\Section{Homogenization of the Vlasov equation with a strong 
external magnetic field}\label{secHVE}
\subsection{Two scale limit of the Vlasov equation}\label{secTSV}

Let us derive first the classical a priori estimate that are available for
the Vlasov equation (\ref{vlasov}).

\begin{lem}\label{estimates}
Under assumption (\ref{ic1}) on the initial condition $f_0$
and (\ref{cvge}),(\ref{cvgb}) on the fields, there exists 
a constant $C$ independent of $\eps$ such that the solution 
$f^\eps$ of the Vlasov equation (\ref{vlasov})  satisfies:
\be\label{estpri}
 \|f^\eps\|_{L^\infty(0,T;L^2(\Omega))}\leq C,
\ee
\end{lem}

\begin{proof}
We multiply the Vlasov equation (\ref{vlasov}) by
$f^\eps$ and easily get
$$\frac{d}{dt}\intom (f^\eps)^2\,d\xvec\,d\vvec +\intom v\cdot\nabla_x 
(f^\eps)^2\,d\xvec\,d\vvec + \intom (\evec^\eps +\vvec\times(\bvec^\eps
+\frac{\mvec}{\eps})) \cdot\nabla_v (f^\eps)^2\,d\xvec\,d\vvec =0.$$
Integrating the second and third terms by parts, they vanish. Hence
$$\frac{d}{dt}\intom (f^\eps)^2\,d\xvec\,d\vvec=0,$$
which means that the $L^2$ norm of $f^\eps$ is conserved and gives us the
estimate on  $f^\eps$ thanks to the bound on the initial condition $f_0$.
\end{proof}

From the a priori estimate obtained in lemma \ref{estimates} we deduce that up
to a subsequence, still denoted by $\eps$:
$$f^\eps \rightharpoonup f \mbox{ in } L^\infty(0,T;L^2(\Omega)), 
\mbox{ weak}-*.$$
Applying the result of N'Guetseng \cite{nguetseng:1989} and Allaire
\cite{allaire:1992} exposed in the Introduction (see Theorem \ref{N'GA})
we may deduce that for every period $\theta$, there exists a
$\theta$-periodic profile $F_\theta(t,\tau,\xvec,\vvec)\in L^\infty(0,T;
L_\theta^\infty(\rit_\tau; L^2(\Omega)))$ such that for all
$\psi_\theta(t,\tau,\xvec,\vvec)$ regular, with compact support with respect
to $(t,\xvec,\vvec)$ and $\theta$-periodic with respect to $\tau$ we have,
up to a subsequence,
\be\label{tsv1}
\intq f^\eps\psi_\theta^\eps\,dt\,d\xvec\,d\vvec\rightarrow \intq
\int_0^\theta F_\theta\psi_\theta\,d\tau\,dt\,d\xvec\,d\vvec.
\ee

The two scale limit of equation (\ref{vlasov}) is led in three steps. In the
first one, using a weak formulation of (\ref{vlasov}) with oscillating test
functions (see Tartar \cite{tartar:1977} and Bensoussan, Lions and Papanicolaou
\cite{bensoussan/lions/papanicolaou:1978}), and using the convergence 
result (\ref{tsv1}), we obtain a
constraint equation for the $\theta$-periodic profile $F_\theta$, for every
$\theta$. The second step is devoted to the resolution of this constraint
equation. This will lead to the natural period $\theta=2\pi$ for the profile.
In the third one, using ad hoc oscillating test functions we deduce the
equation satisfied by F.

\textit{Step 1.} Deduction of the constraint equation: Multiplying the
Vlasov equation (\ref{vlasov}) by
$\psi_\theta^\eps\equiv\psi_\theta(t,\frac{t}{\eps},\xvec,\vvec)$ with
$\psi_\theta(t,\tau,\xvec,\vvec)$ regular, with compact support in
$(t,\xvec,\vvec)\in\calq$ and $\theta$-periodic in $\tau\in\rit_\tau$, 
and integrating by parts using that 
$div_v(\evec^{\eps}+\vvec\times(\bvec^{\eps}+\frac{\mvec}{\eps}))=0$,
we get 
\be\label{tsv3}
\begin{array}{r}
   \ds
\intq f^\eps[(\fracp{\psi_\theta}{t})^\eps 
+\frac{1}{\eps} (\fracp{\psi_\theta}{\tau})^\eps
+\vvec\cdot(\nabla_x \psi_\theta)^\eps
+(\evec^{\eps}+\vvec\times(\bvec^{\eps}+\frac{\mvec}{\eps}))
\cdot(\nabla_v\psi_\theta)^\eps]\,dt\,d\xvec\,d\vvec =
\\ \ds
 -\intom f_0
(\psi_\theta)_{|t=0}^\eps \,d\xvec\,d\vvec.
\end{array}
\ee
Notice that  $\calq=[0,T)\times\Omega$ is not an open set. Consequently,
the contribution of the initial data stand in the left hand side of 
(\ref{tsv3}). \newline
Multiplying (\ref{tsv3}) by $\eps$ and passing to the limit, applying
(\ref{tsv1}), we deduce that the $\theta$-periodic profile $F_\theta$
associated with $f^\eps$ satisfies
\be\label{tsv4}\left\{\ba{l}\ds
\intq\int_0^\theta F_\theta(\fracp{\psi_\theta}{\tau}+(\vvec\times\mvec)\cdot
\nabla_v\psi_\theta)\,d\tau\,dt\,d\xvec\,d\vvec=0,
\\ \ds
F_\theta \mbox{ is $\theta-$periodic in $\tau$.}
\ea\right.
\ee
As this is realized for every $\psi_\theta$ regular, compactly supported in
$(t,\xvec,\vvec)\in \calq$ and $\theta$-periodic in $\tau\in\rit_\tau$, 
it is straightforward to 
show
that (\ref{tsv4}) is equivalent to 
\be\label{tsv5}\left\{\ba{l}\ds
\intrv \int_0^\theta F_\theta(\fracp{\psi_\theta}{\tau}+
(\vvec\times\mvec)\cdot
\nabla_v\psi_\theta)\,d\tau\,dt\,d\vvec=0,
\\ \ds
F_\theta \mbox{ is $\theta-$periodic in $\tau$,}
\ea\right.
\ee
for every $\psi_\theta\equiv\psi_\theta(\tau,\vvec)$ regular, with compact
support in $\vvec$ and $\theta$-periodic in $\tau$, for almost every
$(\xvec,t)$, and also equivalent to
\be\label{tsv6}\left\{\ba{l}\ds
\fracp{F_\theta}{\tau}+(\vvec\times\mvec)\cdot\nabla_v F_\theta=0
\mbox{ in } {\cal D}'(\rit_\tau\times\rit_v^3),
\\ \ds
F_\theta \mbox{ is $\theta-$periodic in $\tau,$}
\ea\right.
\ee
for almost every $(\xvec,t)$. The two
equivalent forms (\ref{tsv5}) and (\ref{tsv6}) will be used in the sequel.
\newline Since the solution of this equation is not unique, 
we call it "constraint equation", and the goal is 
now to derive the form this equation imposes to $F_\theta$.

\textit{Step 2.} Consequences of the constraint equation: Intuitively, the
constraint equation means that $F_\theta$ is constant along the characteristics
of the dynamical system $\dot{V}=V\times\mvec$. As we shall see soon, those
characteristics are helicoids around the magnetic vector $\mvec$. Hence we first
introduce $\uvec(\vvec,\tau)$ the transformation $\vvec\in \rit^3\rightarrow
\uvec(\vvec,\tau)\in \rit^3$ letting invariant the projection of $\vvec$ onto
$\mvec$ and rotating of an angle $\tau$ its projection onto the plane
orthogonal to $\mvec$. We have:
\be\label{defu}
\uvec(\vvec,\tau)=v_1\basu+[v_2 \cos\tau- v_3 \sin\tau]\basd +[v_2 \sin\tau +
v_3 \cos\tau]\bast.\ee

Forgetting for the time being the periodicity ondition let us see what the
constraint equation yields. Consider $F(\tau,\vvec)$ solution of the following
equation
\be\label{tsv7}
\fracp{F}{\tau}+(\vvec\times\mvec)\cdot\nabla_v F=0
\mbox{ in } {\cal D}'(\rit_\tau\times\rit_v^3).\ee

\begin{lem}\label{lem4.1} If  $F(\tau,\vvec)\in
L^\infty(\rit_\tau,L^2(\rit_v^3))$ satisfies (\ref{tsv7}), then there exists a
function $G\in L^2(\rit_u^3)$ such that
\be\label{tsv8}
F(\tau,\vvec)=G(\uvec(\vvec,\tau)).\ee
\end{lem}

\begin{proof} Following Raviart \cite{raviart:1985}, where weak solutions of
first order hyperbolic equations are derived using their characteristic 
equations, $F$ satisfies (\ref{tsv7}) if and only if it is the translation
along the characteristics $V(\tau;\vvec,s)$ solution of
\be\label{tsv9}
\left\{\begin{array}{l}
\displaystyle\frac{dV}{d\tau} = V\times\mvec, \vspace{2mm}\\
V(s;\vvec,s) = \vvec,
\end{array}\right.\ee
of a function $G$. This means
$$F(\tau,\vvec)=G(V(0;\vvec,\tau)).$$

\begin{rem}
In (\ref{tsv9}) $V(\tau;\vvec,s)$ means the solution at time $t=\tau$ taking
the value $\vvec$ at time $t=s$.
\end{rem}
\noindent As $\mvec$ is the first basis vector $\basu$ of the basis 
$(\basu,\basd,\bast)$
of $\rit_v^3$ an easy computation yields the solution of (\ref{tsv9}):
\be\label{tsv9bis}
V(\tau;\vvec,s)= v_1 \basu+[v_2\cos(\tau-s) +v_3\sin(\tau-s)]\basd +
[-v_2\sin(\tau-s) +v_3\cos(\tau-s)]\bast.\ee
So we have
$$V(0;\vvec,\tau)=\uvec(\vvec,\tau),$$
where $\uvec(\vvec,\tau)$ is defined by (\ref{defu}). And
$F(\tau,\vvec)=G(\uvec(\vvec,\tau))$.
\end{proof}

Now we take the periodicity condition under consideration.
First, because of the previous Lemma and since 
$\uvec(\vvec,\tau)$ is $2\pi$-periodic in $\tau$, if the period
$\theta$ is incommensurable with $2\pi$ a fonction $F(\tau,\vvec)$
solution of (\ref{tsv7}) and $\theta-$periodic in $\tau$ is constant.
As a consequence, the $\theta-$periodic profile $F_ \theta$ contains
no information on the oscillations of $(f^\eps).$ \newline
Beside this, if $\theta$ is (a multiple of) $2\pi,$ a fonction 
$F(\tau,\vvec)$ satisfying (\ref{tsv7}) naturally satisfy the periodicity
condition. More precisely we have the
\begin{lem}\label{lmlm2.3}
A function $F(\tau,\vvec)\in L_{2\pi}^\infty(\rit_\tau, L^2(\rit_v^3))$
satisfies
$$\fracp{F}{\tau}+(\vvec\times\mvec)\cdot\nabla_v F=0
\mbox{ in } {\cal D}'(\rit_\tau\times\rit_v^3),$$
if and only if there exists a function $G\in L^2(\rit_u^3)$ such that
\be\label{tsv10000000000}
F(\tau,\vvec)=G(\uvec(\vvec,\tau)).\ee
\end{lem}
Hence, among every possible profile, we are incited to select the 
$2\pi-$ periodic one.

As a conclusion of this step, applying Lemma \ref{lmlm2.3}, 
the $2\pi$-periodic
profile $F(t,\tau,\xvec,\vvec)=$ $F_{2\pi}(t,\tau,\xvec,\vvec)$ associated
with the sequence $f^\eps,$ solutions of (\ref{vlasov}), writes
\be\label{tsv10}
F(t,\tau,\xvec,\vvec) = G(t,\xvec,\uvec(\vvec,\tau)),
\ee
for a function $G \in L^\infty(0,T; L^2(\rit_x^3\times\rit_u^3)).$

\textit{Step 3.} Equation satisfied by $G$: We now look for the equation
satisfied by $G$ linked to $F$ by (\ref{tsv10}). We have
\begin{lem}\label{lemma2.4ret}
The function $G(t,\xvec,\uvec)$ linked to the  $2\pi$-periodic profile
$F$ by (\ref{tsv10}) is the unique solution of:
\be\label{tsv11}
\begin{array}{l}
   \ds
\fracp{G}{t}+\uvec_\parallel\cdot\nabla_x G
+(\evec_\parallel+\uvec\times\bvec_\parallel)\cdot\nabla_u G =0,
\\ \ds
G_{|t=0} = \frac{1}{2\pi} f_0,
\end{array}
\ee
where for any vector field $\uvec=u_1\basu+u_2\basd+u_3\bast$, we define
$\uvec_\parallel=u_1\basu$ its projection onto the direction of the external
magnetic field $\mvec$.
\end{lem}

\begin{proof}
For any regular function $\varphi$, 
let us consider the regular and $2\pi$-periodic in $\tau$ function 
$\psi(t,\tau,\xvec,\vvec)=$ $\varphi(t,\xvec,\uvec(\vvec,\tau))$ 
which satisfies
\be\label{tsv12}
\fracp{\psi}{\tau} +(\vvec\times\mvec)\cdot \nabla_v\psi =0.\ee
Now take
$\psi^\eps\equiv \psi(t,\frac{t}{\eps},\xvec,\vvec)$ as a test 
function in the
weak formulation (\ref{tsv3}) of the Vlasov equation (\ref{vlasov}). 
Then due to
(\ref{tsv12}), we have
\be\label{tsv13}
\intq f^\eps[(\fracp{\psi}{t})^\eps +\vvec\cdot(\nabla_x\psi)^\eps
+(\evec^{\eps}+\vvec\times\bvec^{\eps})
\cdot(\nabla_v\psi)^\eps]\,dt\,d\xvec\,d\vvec = -\intom f_0
(\psi)_{|t=0}^\eps \,d\xvec\,d\vvec.\ee
Passing to the two scale
limit, using that $\evec^\eps$ and $\bvec^\eps$ converge strongly, yields
\be\label{tsv14}
\intq \int_0^{2\pi} F[\fracp{\psi}{t} +\vvec\cdot\nabla_x\psi
+(\evec+\vvec\times\bvec)\cdot(\nabla_v\psi)]\,d\tau\,dt\,d\xvec\,d\vvec = 
-\intom f_0 \psi(0,0,\xvec,\vvec)\,d\xvec\,d\vvec,\ee
which is equivalent to
\bea\label{tsv15}
\intq \int_0^{2\pi} G(t,\xvec,\uvec(\vvec,\tau))
[\fracp{\varphi}{t}(t,\xvec,\uvec(\vvec,\tau))
+\vvec\cdot\nabla_x\varphi(t,\xvec,\uvec(\vvec,\tau)) \hspace{3cm}&&\nn\\
+(\evec+\vvec\times\bvec)\cdot\nabla_v\{\varphi(t,\xvec,\uvec(\tau,\vvec))\}]
\,d\tau\,dt\,d\xvec\,d\vvec =  -\intom f_0
\varphi(0,\xvec,\uvec(\vvec,0))\,d\xvec\,d\vvec.\eea
Making the change of variables $\uvec=\uvec(\vvec,\tau)$, noticing
that $\vvec= V(\tau;\uvec,0)$ with the notations of (\ref{tsv9}) and
that $d\uvec=d\vvec$, we get on the left-hand-side of (\ref{tsv15})
\bea\label{tsv16}
\intq \int_0^{2\pi} G(t,\xvec,\uvec)[\fracp{\varphi}{t}(t,\xvec,\uvec)
+[u_1\basu+[u_2 \cos\tau+ u_3 \sin\tau]\basd \hspace{3cm}&&\nn\\
+[-u_2 \sin\tau +
u_3 \cos\tau]\bast]\cdot\nabla_x\varphi(t,\xvec,\uvec)
+(\evec+[u_1\basu+[u_2 \cos\tau+ u_3 \sin\tau]\basd &&\nn\\
+[-u_2 \sin\tau +
u_3 \cos\tau]\bast]\times\bvec)\cdot[\fracp{\varphi}{u_1}\basu
+[\cos\tau\fracp{\varphi}{u_2}+\sin\tau\fracp{\varphi}{u_3}]\basd &&\nn\\
+[-\sin\tau\fracp{\varphi}{u_2}+\cos\tau\fracp{\varphi}{u_3}]\bast]
\,d\tau\,dt\,d\xvec\,d\uvec.\eea
Now as neither $G$ nor $\varphi$ depend on $\tau$, let us perform the
integration with respect to $\tau$. This yields, dividing by $2\pi$, 
\bean
&&\intq G[\fracp{\varphi}{t}+u_1\basu\cdot\nabla_x\varphi
+(E_1\fracp{\varphi}{u_1}+B_1(u_3\fracp{\varphi}{u_2}
- u_2\fracp{\varphi}{u_3})]\,dt\,d\xvec\,d\uvec \\
&&~~~~~~~~~~~~~~~~~~~~
 =-\frac{1}{2\pi}\intom f_0\varphi(0,\xvec,\uvec))\,d\xvec\,d\uvec.
\eean
as the terms in $\cos\tau$, $\sin\tau$ and $\cos\tau\sin\tau$ vanish when
integrated between 0 and $2\pi$. This gives us the equation verified by $G$ in
the sense of distributions.
\end{proof}

The uniqueness of the solution of (\ref{tsv11}) enables to deduce that 
the whole sequence $f^\eps$ two scale converges to $F$ and,
because of the link (\ref{tsv2}) between $F$ and $f,$
weakly$-*$ converges to $f.$

\subsection{Proof of Theorems \ref{Th1.1} and \ref{Th1.2}}\label{secHV}
Replacing first $\uvec$ by $\uvec(\vvec,\tau)$ in (\ref{tsv11}) we have
\be\label{ajf0}
\begin{array}{l}
   \ds
\fracp{G}{t}(t,\xvec,\uvec(\vvec,\tau))
+\uvec_\parallel(\vvec,\tau)\cdot\nabla_x G(t,\xvec,\uvec(\vvec,\tau))
\hspace{2cm} \\ \ds ~\hfill
+(\evec_\parallel+\uvec(\vvec,\tau)\times\bvec_\parallel)
\cdot(\nabla_u G)(t,\xvec,\uvec(\vvec,\tau)) =0,
\\ \ds
G(0,\xvec,\uvec(\vvec,\tau))=\frac{1}{2\pi}f_0(\xvec,\uvec(\vvec,\tau)).
\end{array}
\ee
Secondly, we get 
\be\label{ajf1}
(\evec_\parallel+\uvec(\vvec,\tau)\times\bvec_\parallel)
\cdot(\nabla_u G)(t,\xvec,\uvec(\vvec,\tau)) =
(\evec_\parallel+\vvec\times\bvec_\parallel)
\cdot\nabla_v (G(t,\xvec,\uvec(\vvec,\tau))).
\ee
Indeed, straightforward computations give
$$
\begin{array}{l}
   \ds
E_1 \fracp{}{v_1}(G(t,\xvec,\uvec(\vvec,\tau))) =
E_1 \fracp{G}{u_1}(t,\xvec,\uvec(\vvec,\tau)), 
\\ \ds
B_1u_3\fracp{}{v_2}(G(t,\xvec,\uvec(\vvec,\tau)))=
B_1(v_2\sin\tau\cos\tau\fracp{G}{u_2} +
v_3\cos^2\tau\fracp{G}{u_2}+\hspace{20mm}
\\ \ds\hfill
v_2 \sin^2\tau\fracp{G}{u_3}+v_3\sin\tau\cos\tau\fracp{G}{u_3}),
\\ \ds
B_1u_2\fracp{}{v_2}(G(t,\xvec,\uvec(\vvec,\tau)))=
-B_1(v_2\sin\tau\cos\tau\fracp{G}{u_2} +
v_3\sin^2\tau\fracp{G}{u_2}+~~~
\\ \ds\hfill
v_2 \cos^2\tau\fracp{G}{u_3}-v_3\sin\tau\cos\tau\fracp{G}{u_3}),
\end{array}
$$
and summing up these three relations yields (\ref{ajf1}).
Hence (\ref{ajf0}) writes
\be\label{ajf00}
\begin{array}{l}
   \ds
\fracp{G}{t}(t,\xvec,\uvec(\vvec,\tau))
+\uvec_\parallel(\vvec,\tau)\cdot\nabla_x G(t,\xvec,\uvec(\vvec,\tau))
\hspace{2cm} \\ \ds ~\hfill
+(\evec_\parallel+\vvec\times\bvec_\parallel)
\cdot\nabla_v (G(t,\xvec,\uvec(\vvec,\tau)))=0,
\\ \ds
G(0,\xvec,\uvec(\vvec,\tau))=\frac{1}{2\pi}f_0(\xvec,\uvec(\vvec,\tau)),
\end{array}
\ee
and since
$$
F(t,\tau,\xvec,\vvec)=G(t,\xvec,\uvec(\vvec,\tau)),
$$
we obtain equation (\ref{vlastwoscale}.b,c). 
As (\ref{vlastwoscale}.a) is only the constraint, Theorem 
\ref{Th1.2} is proved.
{\hfill\rule{2mm}{2mm}\vskip3mm \par}

Then, in order to achieve the proof of Theorem \ref{Th1.2}, we shall
deduce the equation satisfied by $f$ using the integral relation
(\ref{tsv2}) linking $F$ and $f.$ 


\noindent As
$$
f(t,\xvec,\vvec) = \int_0^{2\pi}F(t,\tau,\xvec,\vvec)\,d\tau=
\int_0^{2\pi}G(t,\xvec,\uvec(\vvec,\tau))\,d\tau,
$$
integrating (\ref{ajf00}) in $\tau,$ gives
\be\label{ajf2}
\begin{array}{l}
   \ds
\fracp{}{t}\left(\int_0^{2\pi}G(t,\xvec,\uvec(\vvec,\tau))\,d\tau \right)
+\vvec_\parallel\cdot\nabla_x 
\left(\int_0^{2\pi}G(t,\xvec,\uvec(\vvec,\tau))\,d\tau \right)\hspace{20mm}
\\ \ds\hfill
+(\evec_\parallel+\vvec\times\bvec_\parallel)
\cdot \nabla_v \left(\int_0^{2\pi}G(t,\xvec,\uvec(\vvec,\tau))\,d\tau \right)
=0,
\\ \ds
\int_0^{2\pi}G(0,\xvec,\uvec(\vvec,\tau))\,d\tau=
\frac{1}{2\pi}\int_0^{2\pi}f_0(\xvec,\uvec(\vvec,\tau))\,d\tau,
\end{array}
\ee
giving the homogenized Vlasov equation
(\ref{vlasovlim}). Hence Theorem \ref{Th1.1} is proved
{\hfill\rule{2mm}{2mm}\vskip3mm \par}

\begin{rem}\label{ZobiRem}
The goal of this remark is to show how we may adapt our method to deduce
the equation for the two scale limit in the case of a non uniform strong 
magnetic field.
We assume that $\mvec(t,\xvec)$ is a smooth function from $\calo$
to the set of vectors with norme 1. That implies
\be\label{AA}
\basu = \RRR (t,\xvec) \mvec(t,\xvec),
\ee
where $\RRR (t,\xvec)$ is a smooth map from $\calo$
to the set of orthogonal matrices.
In this framework, the way to deduce the constraint equation
\be\label{BB}
\fracp{F}{\tau} + (\vvec\times\mvec) \cdot\nabla_v F=0,
\ee
remains the same. As $\RRR(t,\xvec)$ is nothing but the matrix of a 
change of coordinates leading $\mvec(t,\xvec)$ onto $\basu$ this 
constraint implies
\be\label{CC}
F(t,\tau,\xvec,\vvec) = 
G(t,\xvec,\RRR^T(t,\xvec) \uvec(\RRR (t,\xvec)\vvec,\tau)),
\ee
where $\RRR^T$ is the transposed (but also the reverse) matrix of $\RRR.$
In order to give a more usable form to (\ref{CC}), we denote
by
\be\label{DD}
\rvec(\tau) = \left(\ba{ccc}
1&0&0\\
0&\cos\tau&-\sin\tau\\
0&\sin\tau&\cos\tau
              \ea\right),
\ee
and since $\RRR^T \uvec(\RRR\vvec,\tau))= \RRR^T\rvec(\tau) \RRR\vvec,$
(\ref{CC}) also writes
\be\label{EE}
F(t,\tau,\xvec,\vvec) = 
G(t,\xvec,\RRR^T(t,\xvec)\rvec(\tau)\RRR (t,\xvec)\vvec).
\ee
Now, in order to get the equation satisfied by $G,$ we proceed in a similar 
way as in the proof of Lemma \ref{lemma2.4ret}, with this difference that
we use test functions
\be\label{FF}
\psi(t,\tau,\xvec,\vvec)=
\varphi(t,\xvec,\RRR^T(t,\xvec)\rvec(\tau)\RRR (t,\xvec)\vvec).
\ee
Those test functions naturally satisfy the constraint and the computation 
leading to equation (\ref{tsv14}) remains valid. But here,
\be\label{GG}
\fracp{\psi}{t}=\fracp{\varphi}{t}+ 
(\fracp{\RRR^T}{t}\rvec(\tau)\RRR \vvec+
\RRR^T \rvec(\tau)\fracp{\RRR}{t}\vvec)\cdot\nabla_u \varphi,
\ee
\be\label{GGa}
\fracp{\psi}{x_i}=\fracp{\varphi}{x_i}+ 
(\fracp{\RRR^T}{x_i}\rvec(\tau)\RRR \vvec+
\RRR^T \rvec(\tau)\fracp{\RRR}{x_i}\vvec)\cdot\nabla_u \varphi,
\ee
and 
\be\label{HH}
\nabla_v \psi = (\RRR^T\rvec(\tau)\RRR)^T\nabla_u \varphi =
(\RRR^T\rvec(-\tau)\RRR)\nabla_u \varphi.
\ee
Hence, in place of formula (\ref{tsv16}), making the change of variables
$\uvec=\RRR^T\rvec(\tau)\RRR\vvec,$
$(\vvec=\RRR^T\rvec(-\tau)\RRR\uvec)$, we get
\bea\label{II}
\intq \int_0^{2\pi} G(t,\xvec,\uvec)[\fracp{\varphi}{t}+ 
(\fracp{\RRR^T}{t}\RRR \uvec+
\RRR^T \rvec(\tau)\fracp{\RRR}{t}\RRR^T\rvec(-\tau)\RRR\uvec)
\cdot\nabla_u \varphi~~~~~~~~~~~~~~~~ &&\nn\\
+\RRR^T\rvec(-\tau)\RRR\uvec \cdot (\nabla_x\varphi
+\left(\ba{c} 
(\fracp{\RRR^T}{x_1}\RRR \uvec+
\RRR^T \rvec(\tau)\fracp{\RRR}{x_1}\RRR^T\rvec(-\tau)\RRR\uvec)
\cdot\nabla_u \varphi
\\
(\fracp{\RRR^T}{x_2}\RRR \uvec+
\RRR^T \rvec(\tau)\fracp{\RRR}{x_2}\RRR^T\rvec(-\tau)\RRR\uvec)
\cdot\nabla_u \varphi
\\ 
(\fracp{\RRR^T}{x_3}\RRR \uvec+
\RRR^T \rvec(\tau)\fracp{\RRR}{x_3}\RRR^T\rvec(-\tau)\RRR\uvec)
\cdot\nabla_u \varphi
\ea\right) ) &&\nn\\
+\evec\cdot\RRR^T\rvec(-\tau)\RRR\nabla_u \varphi
+((\RRR^T\rvec(-\tau)\RRR\uvec)\times \bvec) \cdot
(\RRR^T\rvec(-\tau)\RRR\nabla_u \varphi)]
\,d\tau\,dt\,d\xvec\,d\uvec &&\nn\\
~~~~~~~~~~~~~~~~~~~~~~~~~~=
-\intom f_0\varphi(0,\xvec,\uvec))\,d\xvec\,d\uvec.
\eea
Since 
$$
\evec\cdot\RRR^T\rvec(-\tau)\RRR\nabla_u \varphi= 
\RRR^T\rvec(\tau)\RRR\evec\cdot\nabla_u \varphi,
$$
and
$$
((\RRR^T\rvec(-\tau)\RRR\uvec)\times \bvec) \cdot
(\RRR^T\rvec(-\tau)\RRR\nabla_u \varphi)=
\RRR^T\rvec(\tau)\RRR\BB\RRR^T\rvec(-\tau)\RRR\uvec \cdot\nabla_u \varphi,
$$
with 
$$
\BB = \left( \ba{ccc}
0&B_3&-B_2\\
-B_3&0&B_1\\
B_2&-B_1&0
      \ea \right),
$$
we finally get, integrating (\ref{II}) in $\tau$ and defining
$$
\ovr = \left( \ba{ccc}
1&0&0\\
0&0&0\\
0&0&0
      \ea \right),
$$
\bea\label{LL}\ds
\fracp{G}{t}+ \RRR^T\ovr\RRR\uvec\cdot\nabla_xG
+(\RRR^T\ovr\RRR\evec+
\frac{1}{2\pi}\int_0^{2\pi}\RRR^T\rvec(\tau)\RRR\BB\RRR^T\rvec(-\tau)\RRR
d\tau \;\uvec)\cdot\nabla_u G~~~~~~~~~~~ &&\nn\\
+\nabla_u \cdot [(\fracp{\RRR^T}{t}\RRR \uvec+ \frac{1}{2\pi}\int_0^{2\pi}
\RRR^T \rvec(\tau)\fracp{\RRR}{t}\RRR^T\rvec(-\tau)\RRR d\tau \;\uvec)G]
~~~~~~~~~~~~~~~~~&&\nn\\
+\nabla_u \cdot [(\frac{1}{2\pi}\int_0^{2\pi}
\RRR^T\rvec(-\tau)\RRR\uvec\cdot 
\left(\ba{c} 
(\fracp{\RRR^T}{x_1}\RRR \uvec+
\RRR^T \rvec(\tau)\fracp{\RRR}{x_1}\RRR^T\rvec(-\tau)\RRR\uvec)G
\\
(\fracp{\RRR^T}{x_2}\RRR \uvec+
\RRR^T \rvec(\tau)\fracp{\RRR}{x_2}\RRR^T\rvec(-\tau)\RRR\uvec)G
\\ 
(\fracp{\RRR^T}{x_3}\RRR \uvec+
\RRR^T \rvec(\tau)\fracp{\RRR}{x_3}\RRR^T\rvec(-\tau)\RRR\uvec)G
\ea\right) )d\tau]=0,&&\nn\\
G_{|t=0} = \frac{1}{2\pi}f_0.\hspace{120mm}
\eea
where for 4 vectors of $\rit^3,$ $A,B,C,D$ we denote 
$$
A\cdot \left(\ba{c} B\\C\\D\ea\right) = A_1B+A_2C+A_3D.
$$
This last equation is the equation for the two scale limit
when the strong magnetic field is non uniform.
\end{rem}

\Section{Homogenization of the Vlasov equation with strong 
external magnetic and electric fields}\label{secSME}

We study here a variant of the previous problem. We  
homogenize the following Vlasov equation with strong 
and constant external magnetic and electric fields:
\be\label{vlasovbis}\left\{\ba{l}
\ds\fracp{f^{\eps}}{t}+\vvec\cdot\nabla_xf^{\eps}+
((\evec^{\eps}+\frac{\nvec}{\eps})+
\vvec\times(\bvec^{\eps}+\frac{\mvec}{\eps}))
\cdot\nabla_vf^{\eps} =0,\\
\ds f^\eps_{|t=0}=f_0,
\ea\right.\ee
for constant vectors $\mvec \in \sit^2,$ $\nvec \in \sit^2,$
$\mvec\perp\nvec$ and under the same assumptions 
(\ref{ic1})-(\ref{cvgb}) as previously.
With no lost of generality, we set
$\mvec=\basu$ and $\nvec=\basd.$ 

\noindent Since the a priori estimate (\ref{estpri}) remains valid
we always have
$$
f^\eps \rightharpoonup  f \mbox{ in } L^\infty(0,T;L^2(\Omega)),
 \mbox{ weak}-*.
$$
And, there exists a
$2\pi$-periodic profile $F(t,\tau,\xvec,\vvec)\in L^\infty(0,T;
L_{2\pi}^\infty(\rit_\tau; L^2(\Omega)))$ such that for all
$\psi(t,\tau,\xvec,\vvec)$ regular, with compact support with 
respect to $(t,\xvec,\vvec)$ and $2\pi$-periodic with respect 
to $\tau$ we have
\be\label{tsv1bis}
\intq f^\eps\psi^\eps\,dt\,d\xvec\,d\vvec\rightarrow
\intq
\int_0^{2\pi} F\psi\,d\tau\,dt\,d\xvec\,d\vvec.
\ee

Proceeding as in section \ref{secTSV} we obtain the
following weak formulation with oscillating test functions:
\be\label{tsv3bis}
\begin{array}{l}
   \ds
\intq f^\eps[(\fracp{\psi}{t})^\eps 
+\frac{1}{\eps} (\fracp{\psi}{\tau})^\eps
+\vvec\cdot(\nabla_x \psi)^\eps \\
   \ds ~~~~
+((\evec^{\eps}+\frac{\nvec}{\eps})+\vvec\times(\bvec^{\eps}+\frac{\mvec}{\eps}))
\cdot(\nabla_v\psi)^\eps]\,dt\,d\xvec\,d\vvec =
 -\intom f_0
(\psi)_{|t=0}^\eps \,d\xvec\,d\vvec,
\end{array}
\ee
which, passing to the limit in $\eps$ yields
\be\label{tsv6bis}
\fracp{F}{\tau}+(\nvec+\vvec\times\mvec)\cdot\nabla_v F=0
\mbox{ in } {\cal D}'(\rit_\tau\times\rit_v^3).
\ee
This constraint equation means that 
$$F(t,\tau,\xvec,\vvec)=G(t,\xvec,V(0;\vvec,\tau)),$$
where  $V(\tau;\vvec,s)$ is solution of
\be\label{tsv9biss}
\left\{\begin{array}{l}
\displaystyle\frac{dV}{d\tau} =
\nvec+V\times\mvec, \vspace{2mm}\\
V(s;\vvec,s) = \vvec.
\end{array}\right.\ee
We have
\be\label{tsv9bisbis}
V(\tau;\vvec,s)= v_1 \basu+
[v_2\cos(\tau-s) +(v_3+1)\sin(\tau-s)]\basd +
[-v_2\sin(\tau-s) +(v_3+1)\cos(\tau-s)-1]\bast,\ee
and then
\be\label{tronchemoi}
V(0;\vvec,\tau)=\uvec(\vvec,\tau) =v_1 \basu+
[v_2\cos\tau -(v_3+1)\sin\tau]\basd +
[+v_2\sin\tau +(v_3+1)\cos\tau-1]\bast.
\ee
Hence we conclude
$$F(t,\tau,\xvec,\vvec)=G(t,\xvec,\uvec(\vvec,\tau)),$$
for a function $G \in L^\infty(0,T; 
L^2(\rit_x^3\times\rit_u^3)).$

\begin{rem}
If $\nvec$ were not orthogonal to $\mvec,$ i.e. if $\nvec=
n_1\basu+\basd,$ the consequence of equation (\ref{tsv6bis})
would be $F(t,\tau,\xvec,\vvec)=G(t,\xvec,\uvec(\vvec,\tau)+n_1\tau\basu).$
Then the periodicity condition and the $L^2$ in $\vvec$ regularity 
of $F$ would imply $G=F=0.$
\end{rem}
Now in order to deduce the equation $G$ satisfies, we take
any regular function $\varphi$
and we consider the regular and $2\pi$-periodic in $\tau$ function 
$\psi(t,\tau,\xvec,\vvec)=$ $\varphi(t,\xvec,\uvec(\vvec,\tau))$ 
which satisfies
\be\label{tsv12bis}
\fracp{\psi}{\tau} +(\nvec+\vvec\times\mvec)
\cdot \nabla_v\psi =0.
\ee
Hence using this test function in the weak formulation 
(\ref{tsv3bis}), the term containing the constraint disappears
and passing then to the limit gives
\be\label{tsv14bis}
\intq \int_0^{2\pi} F[\fracp{\psi}{t} +\vvec\cdot\nabla_x\psi
+(\evec+\vvec\times\bvec)\cdot(\nabla_v\psi)]\,d\tau\,dt\,d\xvec\,d\vvec = 
-\intom f_0 \psi(0,0,\xvec,\vvec)\,d\xvec\,d\vvec.\ee
Replacing $F$ and $\psi$ by theire expressions, (\ref{tsv14bis})
becomes
\bea\label{tsv15bis}
\intq \int_0^{2\pi} G(t,\xvec,\uvec(\vvec,\tau))
[\fracp{\varphi}{t}(t,\xvec,\uvec(\vvec,\tau))
+\vvec\cdot\nabla_x\varphi(t,\xvec,\uvec(\vvec,\tau)) \hspace{3cm}&&\nn\\
+(\evec+\vvec\times\bvec)\cdot\nabla_v\{\varphi(t,\xvec,\uvec(\tau,\vvec))\}]
\,d\tau\,dt\,d\xvec\,d\vvec =  -\intom f_0
\varphi(0,\xvec,\uvec(\vvec,0))\,d\xvec\,d\vvec.\eea
Making the change of variables $\uvec=\uvec(\vvec,\tau)$,
we get on the left-hand-side of (\ref{tsv15bis})
\be\label{tsv16bis}\ba{l}
\ds\intq \int_0^{2\pi} G(t,\xvec,\uvec)[\fracp{\varphi}{t}(t,\xvec,\uvec)
+[u_1\basu+[u_2 \cos\tau+ (u_3+1) \sin\tau]\basd \hspace{3cm}\nn\\
\ds\hspace{5cm}+[-u_2 \sin\tau +
(u_3+1) \cos\tau-1]\bast]\cdot\nabla_x\varphi(t,\xvec,\uvec)\nn\vspace{1mm}\\
\ds+(\evec+[u_1\basu+[u_2 \cos\tau+ (u_3+1) \sin\tau]\basd
\hspace{3cm} \nn\\
\ds~\hspace{3mm}+[-u_2 \sin\tau +
(u_3+1) \cos\tau-1]\bast]\times\bvec)\cdot[\fracp{\varphi}{u_1}\basu
+[\cos\tau\fracp{\varphi}{u_2}+\sin\tau\fracp{\varphi}{u_3}]\basd
\hspace{5mm} \nn\\
\ds~\hfill+[-\sin\tau\fracp{\varphi}{u_2}+\cos\tau\fracp{\varphi}{u_3}]
\bast]
\,d\tau\,dt\,d\xvec\,d\uvec.\ea\ee
Now as neither $G$ nor $\varphi$ depend on $\tau$, let us perform the
integration with respect to $\tau$. This yields, dividing by $2\pi$, 
\bean
&&\intq G[\fracp{\varphi}{t}+(u_1\basu-\bast)\cdot\nabla_x\varphi
+((E_1-B_2)\fracp{\varphi}{u_1}+B_1((u_3+1)\fracp{\varphi}{u_2}
- u_2\fracp{\varphi}{u_3})]\,dt\,d\xvec\,d\uvec \\
&&\hspace{5cm}
 =-\frac{1}{2\pi}\intom f_0\varphi(0,\xvec,\uvec))\,d\xvec\,d\uvec.
\eean
Hence we proved the following result
\begin{thm}
\label{thth3.1}
Under assumptions (\ref{ic1})-(\ref{cvgb}) and with $\mvec=\basu$ and 
$\nvec=\basd,$
the function $G(t,\xvec,\uvec)$ linked to the  $2\pi$-periodic profile
$F$ associated with the solution of (\ref{vlasovbis}) is the unique
solution of:
\be\label{tsv11bis}
\begin{array}{l}
   \ds
\fracp{G}{t}+\left(\ba{c}u_1\\0\\-1 \ea \right)\cdot\nabla_x G
+\left[\left(\ba{c}E_1-B_2\\0\\0 \ea \right)+
\left(\ba{c}u_1\\u_2\\u_3+1 \ea \right)\times
\left(\ba{c}B_1\\0\\0 \ea \right)\right]\cdot\nabla_u G =0,
\\ \ds
G_{|t=0} = \frac{1}{2\pi} f_0,
\end{array}
\ee
\end{thm}

Now we deduce the equation satisfied by the weak limit $f$ using
the relation linking $F$ and $f.$

\noindent Replacing $\uvec$ by $\uvec(\vvec,\tau)$ in (\ref{tsv11bis}), we have
\be\label{tsv11bisbis}
\begin{array}{l}
   \ds
\fracp{G}{t}(t,\xvec,\uvec(\vvec,\tau))
+\left(\ba{c}u_1(\vvec,\tau)\\0\\-1 \ea \right)\cdot
\nabla_x G(t,\xvec,\uvec(\vvec,\tau)) \\
\hspace{2cm}+\left[\left(\ba{c}E_1-B_2\\0\\0 \ea \right)+
\left(\ba{c}u_1(\vvec,\tau)\\u_2(\vvec,\tau)\\u_3(\vvec,\tau)+1 \ea \right)\times
\left(\ba{c}B_1\\0\\0 \ea \right)\right]
\cdot\nabla_u G(t,\xvec,\uvec(\vvec,\tau)) =0,
\\ \ds
G(0,\xvec,\uvec(\vvec,\tau)) = 
\frac{1}{2\pi} f_0(\xvec,\uvec(\vvec,\tau)).
\end{array}
\ee
Then as
\be\label{eeeeeee}\ba{l}
\ds
\left[\left(\ba{c}E_1-B_2\\0\\0 \ea \right)+
\left(\ba{c}u_1(\vvec,\tau)\\u_2(\vvec,\tau)\\u_3(\vvec,\tau)+1 \ea \right)\times
\left(\ba{c}B_1\\0\\0 \ea \right)\right]
\cdot\nabla_u G(t,\xvec,\uvec(\vvec,\tau)) =
\hspace{2cm}\\ \ds\hfill
\left[\left(\ba{c}E_1-B_2\\0\\0 \ea \right)+
\left(\ba{c}v_1\\v_2\\v_3+1 \ea \right)\times
\left(\ba{c}B_1\\0\\0 \ea \right)\right]
\cdot\nabla_v(G(t,\xvec,\uvec(\vvec,\tau))),
\ea\ee
and as
$$
f(t,\xvec,\vvec) = \int_0^{2\pi}F(t,\tau,\xvec,\vvec)\,d\tau=
\int_0^{2\pi}G(t,\xvec,\uvec(\vvec,\tau))\,d\tau,
$$
we have the

\begin{thm}
\label{thth3.2}
Under assumptions (\ref{ic1})-(\ref{cvgb}) and with $\mvec=\basu$ and 
$\nvec=\basd,$ the weak$-*$ limit $f$ of the sequence of
solutions of (\ref{vlasovbis}) is solution of:
\be\label{tsv11fin}
\begin{array}{l}
   \ds
\fracp{f}{t}+\left(\ba{c}v_1\\0\\-1 \ea \right)\cdot\nabla_x f
+\left[\left(\ba{c}E_1-B_2\\0\\0 \ea \right)+
\left(\ba{c}v_1\\v_2\\v_3+1 \ea \right)\times
\left(\ba{c}B_1\\0\\0 \ea \right)\right]\cdot\nabla_v f =0,
\\ \ds
f_{|t=0} = \frac{1}{2\pi}\int_0^{2\pi} f_0(\xvec, \uvec(\vvec,\tau)) \,d\tau,
\end{array}
\ee
with $\uvec(\vvec,\tau)$ given by (\ref{tronchemoi})
\end{thm}

\begin{rem}
As
$$ \left(\ba{c}0\\0\\-1 \ea \right) = \nvec\times\mvec,$$
the presence of $-1$ on the third componant of the advection vector 
shares a lot with the additional $\nvec\times\mvec$ drift effect 
usually found in the "guiding center" approximation with same order
electric and magnetic fields.
\end{rem}
\Section{Homogenization of the Vlasov-Poisson system with a strong 
external magnetic field}\label{secHVP}

Assume that the initial condition satisfies the hypothesis (\ref{ic7}). 
Then for $\varepsilon >0$ given there exists a solution 
in $L^\infty(0,T;L^1\cap L^2(\Omega) ) \times 
L^\infty(0,T;W^{1,\frac75}(\rit^3_x))$ for any $T\in\rit^+$
of the Vlasov-Poisson system (\ref{vlaspoiss}).
Indeed for a given $\varepsilon >0$ a force field 
associated to a constant magnetic field has simply been added to the system
studied  by Horst and Hunze \cite{horst/hunze} and their result can be extended 
easily to our case. 

Now in order to apply the results of section \ref{secHVE} and pass to the 
two-scale limit of the Vlasov equation, we need to prove that there exists
a sequence of electric fields solution of the Vlasov-Poisson
system (\ref{vlaspoiss}) verifying 
$$\evec^\eps\rightarrow \evec \mbox{ in } 
L^\infty(0,T;L^2_{loc}(\rit^3_x))\mbox{ strong,}$$
for any $T\in \rit^+.$
For this, we shall need some apriori estimates independent on $\eps$.

\begin{lem}\label{estimates2}
Assume that the initial condition $f_0$ is such that
  $f_0\in L^2(\Omega)$
and $\intom f_0|v|^2\,d\xvec d\vvec$ is bounded, then there exists a constant
$C$ independent of $\eps$ such that the solution ($\evec^\eps,
f^\eps$) of the Vlasov-Poisson equations (\ref{vlaspoiss}) satisfies,
for any $T\in\rit^+,$
$$ \|f^\eps\|_{L^\infty(0,T;L^2(\Omega))}\leq C,$$
$$ \|(|v|^2\,f^\eps)\|_{L^\infty(0,T;L^1(\Omega))}\leq C,$$
and moreover
$$\|\rho^\eps(\xvec,t)\|_{L^\infty(0,T;L^{\frac75}(\rit_x^3))}\leq C,$$
$$\|\jvec^\eps(\xvec,t)\|_{L^\infty(0,T;L^{\frac76}(\rit_x^3))}\leq C,$$
where $\rho^\eps(\xvec,t)=\int f^\eps\,d\,\vvec$ and 
$\jvec^\eps(\xvec,t)=\int \vvec f^\eps\,d\,\vvec$.
\end{lem}

\begin{proof}
Let us proceed formally in order to simplify the presentation. All the following
is rigorously verified for regularized solutions and then the bounds, that do
not depend on $\eps$, are conserved when passing to the limit with respect to
the regularizing parameter.

Multiplying the Vlasov equation by $|\vvec|^2$ and integrating
with respect to $ \xvec$ and $\vvec$, we get
\be\label{vlasen}
\frac{d}{dt}\intom f^\eps |\vvec|^2 \,d\vvec\,d\xvec
-2\intrx\jvec^\eps\cdot\evec^\eps\,d\xvec =0.
\ee
Now, integrating the Vlasov equation with respect to $\vvec$ yields the 
continuity equation
\be\label{conteq}
\fracp{\rho^\eps}{t}+\div\jvec^\eps=0.
\ee
Using this, we obtain
\bean
\intrx \jvec^\eps\cdot\evec^\eps\,d\xvec &=& 
-\intrx \jvec^\eps\cdot\nabla u^\eps\,d\xvec = 
\intrx \div\jvec^\eps\, u^\eps\,d\xvec,\\
&=&  -\intrx\fracp{\rho_\eps}{t}\,u^\eps\,d\xvec.
\eean
But the Poisson equation yields
$$\frac{1}{2}\frac{d}{dt}\intrx(\nabla u^\eps)^2\,d\xvec = 
\intrx \fracp{\rho_\eps}{t}\,u^\eps\,d\xvec.$$
Hence (\ref{vlasen}) becomes
$$\frac{d}{dt}(\intom f^\eps |\vvec|^2 \,d\vvec\,d\xvec 
+\intrx(\nabla u^\eps)^2\,d\xvec)=0.$$
Finally, integrating this last equation with respect to $t$ and using the
hypothesis on the initial condition $f_0$, we get the bound on
$\intom f^\eps |\vvec|^2 \,d\vvec\,d\xvec $.

Let us now proceed with the $L^2$ estimate on
$f^\eps$. To this aim we multiply the Vlasov equation  by
$f^\eps$ and easily get
$$\frac{d}{dt}\intom (f^\eps)^2\,d\xvec\,d\vvec +\intom v\cdot\nabla_x 
(f^\eps)^2\,d\xvec\,d\vvec + \intom (\evec^\eps +\vvec\times
\frac{\mvec}{\eps}) \cdot\nabla_v (f^\eps)^2\,d\xvec\,d\vvec =0.$$
Integrating the second and third terms by parts, they vanish. Hence
$$\frac{d}{dt}\intom (f^\eps)^2\,d\xvec\,d\vvec=0,$$
which means that the $L^2$ norm of $f^\eps$ is conserved and gives us the
estimate on  $f^\eps$ thanks to the bound on the initial condition $f_0$.

We now come to the last estimates, namely those on $\rho^\eps$ and $\jvec^\eps$.
Following the idea of Horst \cite{horst:1981}, the estimates for 
$\rho^\eps$ and
$\jvec^\eps$ can be obtained by decomposition of the velocity integral:
$$\rho^\eps(\xvec,t)=\intrv f^\eps\,d\vvec = \int_{|v|<R} f^\eps\,d\vvec +
\int_{|v|>R} f^\eps\,d\vvec, $$
for any $R>0$. Using the Cauchy-Schwartz inequality, we have
$$\int_{|v|<R} f^\eps\,d\vvec \leq (\int_{|v|<R}
(f^\eps)^2\,d\vvec)^\frac{1}{2} (\int_{|v|<R}d\vvec)^\frac{1}{2}
\leq C_1 R^\frac{3}{2}(\intrv (f^\eps)^2\,d\vvec)^\frac{1}{2},$$
and
$$\int_{|v|>R} f^\eps\,d\vvec \leq (\int_{|v|>R} \frac{|v|^2}{R^2}
f^\eps\,d\vvec \leq \frac{1}{R^2}\intrv |v|^2 f^\eps\,d\vvec.$$
Hence, we have for any $R>0$
$$|\rho^\eps(\xvec,t)|\leq C_1 R^\frac{3}{2}(\intrv
(f^\eps)^2\,d\vvec)^\frac{1}{2} + \frac{1}{R^2}\intrv |v|^2 f^\eps\,d\vvec.$$
Taking the $R$ which minimizes the right-hand-side we obtain
$$|\rho^\eps(\xvec,t)|\leq C_2 (\intrv (f^\eps)^2\,d\vvec)^\frac{2}{7} (\intrv
|v|^2 f^\eps\,d\vvec)^\frac{3}{7}, $$
and finally
\bean
\intrx |\rho^\eps(\xvec,t)|^\frac{7}{5}\,d\xvec &\leq & C_3 \intrx  (\intrv
(f^\eps)^2\,d\vvec)^\frac{2}{5} (\intrv |v|^2 f^\eps\,d\vvec)^\frac{3}{5}
\,d\xvec, \\
&\leq & C_3 (\intrxrv (f^\eps)^2\,d\xvec\,d\vvec)^\frac{2}{5} (\intrxrv
|v|^2 f^\eps\,d\xvec\,d\vvec)^\frac{3}{5} ,
\eean
thanks to the H\"older inequality. Now, knowing that the terms on the
right-hand-side are bounded, we have our estimate on $\rho^\eps$.

We can proceed with $\jvec^\eps$ in exactly the same way:
$$|\jvec ^\eps(\xvec,t)|\leq\intrv |\vvec| f^\eps\,d\vvec = \int_{|v|<R} |\vvec|
f^\eps\,d\vvec + \int_{|v|>R} |\vvec| f^\eps\,d\vvec, $$
for any $R>0$. Using the Cauchy-Schwartz inequality, we have
$$\int_{|v|<R} |\vvec| f^\eps\,d\vvec \leq (\int_{|v|<R}
(f^\eps)^2\,d\vvec)^\frac{1}{2} (\int_{|v|<R} R^2 d\vvec)^\frac{1}{2},$$
and
$$\int_{|v|>R} |\vvec| f^\eps\,d\vvec \leq \int_{|v|>R} \frac{|v|^2}{R}
f^\eps\,d\vvec.$$
Here again we can find the minimizing $R$ and use H\"older's inequality to
obtain
$$\intrx |\jvec^\eps(\xvec,t)|^\frac{7}{6}\,d\xvec \leq C_4 (\intrxrv 
(f^\eps)^2\,d\xvec\,d\vvec)^\frac{1}{6} (\intrxrv |v|^2
f^\eps\,d\xvec\,d\vvec)^\frac{5}{6},$$
which gives us the estimate on $\jvec^\eps$.
\end{proof}

Now, on the one hand, thanks to the classical regularizing properties of the
Laplacian, $\rho^\eps(\xvec,t)$ bounded in $L^\infty(0,T;L^{\frac75}(\rit_x^3))$
implies that  $u^\eps$ such that $-\Delta u^\eps =\rho^\eps$ is bounded in
$L^\infty(0,T;W^{2,\frac75}(\rit_x^3))$ and hence $\evec^\eps=-\nabla u^\eps$
is bounded in $L^\infty(0,T;W^{1,\frac75}(\rit_x^3))$.
On the other hand, integrating the Vlasov equation with respect to $\vvec$, 
we get the continuity equation 
$$\fracp{\rho^\eps}{t}+\div\jvec^\eps=0.$$
Hence, as $\jvec^\eps$ is bounded in $L^\infty(0,T;L^{\frac76}(\rit_x^3))$,
$\fracp{\rho^\eps}{t}$ is bounded in $L^\infty(0,T;W^{-1,\frac76}(\rit_x^3))$, 
and as we have
$$-\Delta  \fracp{u^\eps}{t} = \fracp{\rho^\eps}{t},$$
the regularizing properties of the Laplacian now yield $\fracp{u^\eps}{t}$ 
bounded in $L^\infty(0,T;W^{1,\frac76}(\rit_x^3))$, which yields that 
$\fracp{\evec^\eps}{t}$ is bounded in $L^\infty(0,T;L^{\frac76}(\rit_x^3))$. 
\newline Thus as
$$W^{1,\frac75}(\rit_x^3) \subset L^2_{loc}(\rit_x^3) 
\subset L^{\frac76}_{loc}(\rit_x^3),$$
the first injection being compact and the second being continuous, the 
Aubin-Lions Lemma (see for example Lions \cite{lions}) yields that 
the functional space 
$$
{\cal U}= \{\evec\in L^\infty(0,T;W^{1,\frac75}(\rit_x^3)),
\fracp{\evec}{t}\in L^\infty(0,T;L^{\frac76}(\rit_x^3)) \},$$
provided with the usual product norm,
is compactly embedded in 
$L^\infty(0,T;L^2_{loc}(\rit_x^3))$ and consequently, as 
$\evec^\eps$ is bounded
in ${\cal U}$ there exists a subsequence of $\evec^\eps$
which converges strongly in $L^\infty(0,T;L^2_{loc}(\rit_x^3))$. 
\newline Hence we have 
\begin{lem}
Under assumption (\ref{ic7}), extracting a subsequence, the sequence
$(f^\eps, E^\eps)$
$\in L^\infty(0,T;L^1\cap L^2(\Omega) ) \times 
L^\infty(0,T;W^{1,\frac75}(\rit^3_x))$ of solutions of (\ref{vlaspoiss}) 
satisfies:
$$
f^\eps \rightharpoonup f \mbox{ in }
L^\infty(0,T;L^2(\Omega)) \mbox{ weak}-*,
$$
$$
f^\eps \mbox{ two scale converges to } F \in 
L^\infty(0,T;L^\infty_{2\pi}(\rit_\tau;L^1\cap L^2(\Omega))),
$$
$$
\evec^\eps \rightarrow \evec \mbox{ in }
L^\infty(0,T;L^2_{loc}(\rit_x^3))\mbox{ strong.}
$$
Moreover,
$$
\rho^\eps \rightharpoonup \rho \mbox{ in }
L^\infty(0,T;L^{\frac75}(\rit_x^3))\mbox{ weak}-*,
$$
$$
\rho^\eps \mbox{ two scale converges to } \ov\rho \in 
L^\infty(0,T;L^\infty_{2\pi}(\rit_\tau;L^{\frac75}(\rit_x^3))),
$$
$$
\jvec^\eps \rightharpoonup \jvec \mbox{ in }
L^\infty(0,T;L^{\frac76}(\rit_x^3)))\mbox{ weak}-*,
$$
$$
\jvec^\eps \mbox{ two scale converges to } \ov\jvec \in
L^\infty(0,T;L^\infty_{2\pi}(\rit_\tau;L^{\frac76}(\rit_x^3))),
$$
for any $T\in \rit^+.$
\end{lem}
Then, passing to the two scale limit in (\ref{vlaspoiss}) applying
the results of section \ref{secHVE} yields system 
(\ref{vlaspoisstwoscale})-(\ref{defrhobar}).
\newline In order to show that $\ov\rho$ does not depend on
$\tau,$ we multiply the continuity equation (\ref{conteq})
by $\varphi(t,\frac{t}{\eps},x)$ with $\varphi(t,\tau,x)$ regular,
with compact support in $t,x$ and periodic in $\tau,$
we integrate it by part and we multiply it by $\eps.$
Passing then to the limit gives
$$
\fracp{\ov\rho}{\tau} = 0.
$$ 
Now, summing up the above arguments proves Theorem \ref{Th1.4}.
{\hfill\rule{2mm}{2mm}\vskip3mm \par}
\noindent Lastly, exactly as in section \ref{secHVE}, by an integration
with respect to $\tau$ we get Theorem \ref{Th1.3}.
{\hfill\rule{2mm}{2mm}\vskip3mm \par}

\bibliographystyle{plain}
\bibliography{biblio}

\end{document}